\documentclass[12pt]{article}

\usepackage[left=3.5cm,right=3.5cm,top=3cm,bottom=3cm]{geometry}

\usepackage{fancyhdr}
\usepackage{graphicx}
\usepackage{amsmath,amsfonts,amssymb,amsthm}
\usepackage{epsfig,epstopdf,color}
\usepackage{enumitem}
\usepackage[nottoc]{tocbibind}
\usepackage{todonotes}
\usepackage{caption}
\usepackage{subcaption}
\usepackage{tikzit}

\makeatletter
\def\blfootnote{\xdef\@thefnmark{}\@footnotetext}
\makeatother

\usepackage[english,algoruled,lined]{algorithm2e}

\usepackage{tikz}
\usetikzlibrary{arrows,automata}

\usepackage[colorlinks,linkcolor=blue,citecolor=purple]{hyperref}
\usepackage{mathtools}

\setlist{itemsep=3pt,parsep=0pt,topsep=2pt,partopsep=0pt}  

\newcommand{\oldqed}{}
\newcommand{\eofClaim}{\hfill\scalebox{.6}{$\Box$}}

\renewcommand{\le}{\leqslant}
\renewcommand{\ge}{\geqslant} 

\newtheorem{theorem}{Theorem}
\newtheorem{lemma}[theorem]{Lemma}

\newtheorem{corollary}[theorem]{Corollary}
\newtheorem{conjecture}[theorem]{Conjecture}
\newtheorem{question}[theorem]{Question}

\theoremstyle{definition}

\title{Graphs with arbitrary Ramsey number and connectivity}
\author{
Isabel Ahme\footnote{Mathematical Institute, University of Oxford,
Oxford OX2\thinspace6GG, United Kingdom.}~\footnote{Email: isabel.ahme@gmail.com}  \quad \;
Alex Scott\protect\footnotemark[1]~\footnote{Research supported by EPSRC grant EP/X013642/1.}}
\date{}

\begin{document}

\maketitle

\begin{abstract}
    The Ramsey number $r(G)$ of a graph $G$ is the minimum number $N$ such that any red-blue colouring of the edges of $K_N$ contains a monochromatic copy of $G$. Pavez-Signé, Piga and Sanhueza-Matamala proved that for any function $n\leq f(n) \leq r(K_n)$, there is a sequence of connected graphs $(G_n)_{n\in \mathbb{N}}$ with $|V(G_n)|=n$ such that $r(G_n)=\Theta(f(n))$ and conjectured that $G_n$ can additionally have arbitrarily large connectivity. In this note we prove their conjecture.
\end{abstract}

\section{Introduction}
The \emph{Ramsey number} $r(G)$ of a graph $G$ is the smallest natural number $N$ such that every red-blue colouring of the edges of $K_N$ results in a monochromatic copy of $G$. Since the early days of Ramsey theory, determining the growth rates of the Ramsey number of various graph families has been a large focus of the field. Classic results by Erd\H{o}s \cite{Erdos1947lowerbound} and Erd\H{o}s and Szekeres \cite{Erdos1935upperbound} imply that $\sqrt{2}^{n}\leq r(K_n)\leq 4^n$. Despite considerable attention, no improvement had been made to the constants $\sqrt{2}$ and 4, until recently Campos, Griffiths, Morris and Sahasrabudhe \cite{campos2023exponential} showed that there is an $\epsilon >0$ such that $r(K_n)\leq (4-\epsilon)^n$. Thus, the growth rate of $r(K_n)$ is exponential in $n$, but its exact behaviour remains unclear since the gap between the two bounds is large. On the other end of possible growth rates, there has been much interest in understanding which graphs have Ramsey number linear in their number of vertices. Burr and Erd\H{o}s \cite{Burr1973conjecture} conjectured that this is true for all graphs of bounded degeneracy, and this was proved by Lee \cite{Lee2017degeneracy}. 

It is easy to see that the Ramsey number of any graph on $n$ vertices lies between  $n$ and $r(K_n)$.  Pavez-Signé, Piga and Sanhueza-Matamala \cite{pavez2022growthrates} raised the question of what intermediate growth rates can be achieved. They prove that for every function $f$ such that $n\leq f(n)\leq r(K_n)$ for all $n$, there is a sequence of connected graphs such that the Ramsey number satisfies this growth rate to within a multiplicative factor.

\begin{theorem}[\cite{pavez2022growthrates}, Theorem 1]
\label{thm:pavez}
    For every $k\geq 2$, there exists a constant $C>0$ such that, for every function $f\colon \mathbb{N}\rightarrow \mathbb{N}$ with $n\leq f(n)\leq r(K_n)$, there is a sequence $(G_n)_{n\in \mathbb{N}}$ of connected graphs  such that $|V(G_n)|=n$ and $f(n) \leq r(G_n) \leq Cf(n)$ for all sufficiently large $n$. 
\end{theorem}

Note that the problem is much simpler if the graphs are not required to be connected: 
taking an $K_r$ and adding $n-r$ isolated vertices gives an $n$-vertex graph with Ramsey number $\max \{n, r(K_r)\}$, which lies within a constant factor of $f(n)$ if $r$ is chosen correctly.
The examples constructed by Pavez-Signé, Piga and Sanhueza-Matamala are connected but not 2-connected, consisting of a path of suitable length attached at one end to a clique (or more generally a complete multipartite graph). Pavez-Signé, Piga and Sanhueza-Matamala conjectured that it should be possible to get a sequence of graphs with any growth rate and connectivity.

\begin{conjecture}[\cite{pavez2022growthrates}, Conjecture 16]
    For every $k\geq 2$ and every function $f \colon \mathbb{N} \rightarrow \mathbb{N}$ such that $n\leq f(n) \leq r(K_n)$ for all $n$, there is a sequence $(G_n)_{n\in \mathbb{N}}$ of $k$-connected graphs  such that $|V(G_n)|=n$ and $r(G_n)=\Theta(f_n)$. 
\end{conjecture}

In this paper, we prove their conjecture.

\begin{theorem}
\label{thm:arbitraryconnectivity}
    For every $k\geq 2$, there exists $C_k>0$ such that the following holds. For every function $f: \mathbb{N} \rightarrow\mathbb{N}$ such that $n\leq f(n)\leq r(K_n)$ for all $n$, there is a sequence $(G_n)_{n\in \mathbb{N}}$ of $k$-connected graphs such that $|V(G_n)|=n$ and $f(n)\leq r(G_n) \leq C_k(f_n)$ for all $n$.
\end{theorem}

\section{Proof of Theorem \ref{thm:arbitraryconnectivity}}

Throughout this section all logarithms are base 2. We first give some tools required for the proof. 

A graph has \emph{density} $d$ if it has $d{n \choose 2}$ edges.
A lemma of Conlon, Fox and Sudakov \cite{CFS2020concentration} is helpful for controlling the growth of Ramsey numbers for dense graphs.  

\begin{lemma} [\cite{CFS2020concentration}, Lemma 5.5]
\label{lm:vertexdeletion}
    For every $d\in(0,1)$ there is $c>0$ such that for any graph $G$ of density at least $d$ and any vertex $v\in V(G)$,
    \begin{equation*}
        r(G) \leq  \frac{c}{d}\log({1}/{d}) \cdot r(G-v).
    \end{equation*}
\end{lemma} 

As noted in \cite{CFS2020concentration}, \cite{pavez2022growthrates}, the following is an immediate consequence.

\begin{corollary}
\label{lemma:completegraphsbound}
    There exist $c_1,c_2>0$ so that for any $n\geq 1$
    \begin{enumerate}
        \item $r(K_{t})\leq c_1r(K_{t-1})$,
        \item $r(K_{t,t})\leq c_2r(K_{t-1,t-1})$.
    \end{enumerate}
\end{corollary}

We will also need lower bounds on $r(K_t)$ and $r(K_{t,t})$. As mentioned before, we have $r(K_t)\geq 2^{t/2}$ \cite{Erdos1947lowerbound}. A similar probabilistic approach leads to the following well-known lower bound on $r(K_{t,t})$.

\begin{lemma}
\label{lm:compbiplowerbound}
    For $t\geq 1$, $r(K_{t,t})\geq 2^{t/2}$.
\end{lemma}

The next result we require concerns unbalanced colourings. We say that a colouring is $\epsilon$\emph{-balanced} if both colour classes have edge density at least $\epsilon$. Erd\H{o}s and Szemerédi \cite{Erdosszemeredi1972unbalanced} showed that if we know that a colouring is unbalanced then we can guarantee larger monochromatic cliques. 
 
\begin{lemma}[\cite{Erdosszemeredi1972unbalanced}, Theorem 2]
\label{lemma:unbalancedcolourings}
    There exists an absolute constant $a>0$ such that for all $\epsilon \in (0,\frac{1}{2}]$ and any positive integer $N$ the following holds. Any two-colouring of $K_N$ which is not $\epsilon$-balanced contains a monochromatic clique of order $\frac{a}{\epsilon \log \left({1}/{\epsilon}\right)}\log N$.
\end{lemma}

Lastly, we recall the following fact.

\begin{lemma}
\label{lm:digraphcolouring}
    Let $D$ be a digraph with maximum in-degree at most $\Delta$. Then there is a colouring of the vertices of $D$ using at most $2\Delta+1$ colours such that the associated colouring of the underlying graph of $D$ is a proper vertex colouring.
\end{lemma}

\begin{proof}
    We construct an ordering of the vertices of $D$ such that each vertex $v\in V(D)$ has at most $2\Delta$ in-neighbours and out-neighbours preceding it. Indeed, since the maximum in-degree is at most $\Delta$, there must be a vertex $v$ whose out-degree is at most $\Delta$. We put this $v$ last in our ordering and continue recursively. A greedy colouring now uses at most $2\Delta+1$ colours.
\end{proof}

We are now ready to prove Theorem \ref{thm:arbitraryconnectivity}.

\begin{proof}[Proof of Theorem \ref{thm:arbitraryconnectivity}]
Let $f\colon \mathbb{N}\rightarrow \mathbb{N}$ be a given function such that $n\leq f(n)\leq r(K_n)$ for all $n\in \mathbb{N}$. For a fixed, large $n$, we wish to construct a suitable graph $G_n$. We split into two cases, depending on the size of $f(n)$ (note that the cases cover overlapping ranges).

\medskip

\emph{Case 1:} $n\leq f(n) \leq 2^{n/8}$. 
Take the smallest $t$ such that $r(K_{t,t})> f(n)$. We note that by minimality, we have
\begin{equation*}
        r(K_{t-1,t-1})\leq f(n)< r(K_{t,t})\leq c_2r(K_{t-1,t-1})\leq c_2f(n),
\end{equation*}
where $c_2$ is the constant from Lemma \ref{lemma:completegraphsbound}. By Lemma \ref{lm:compbiplowerbound} we have $f(n) \geq r(K_{t-1,t-1})\geq 2^{(t-1)/2}$, so $t\leq 2\log f(n) +1$. As $f(n)\leq 2^{n/8}$, this implies $2t \leq n$. Let $G_n$ be the graph constructed as follows. Take a copy of $K_{t,t}$ and $n-2t$ isolated vertices. Within one of the parts of the copy of $K_{t,t}$ we pick $k$ vertices and add a complete bipartite graph between these $k$ vertices and the $n-2t$ isolated vertices (see Figure \ref{fig:case1construction}). It is easy to check that $|V(G_n)|=n$, $G_n$ is $k$-connected and that $r(G_n)\geq r(K_{t,t})> f(n)$.

    \begin{figure}[b]
        \centering
        
        \begin{subfigure}{0.45\textwidth}
        \centering
            \begin{tikzpicture} [scale=0.9]
	\begin{pgfonlayer}{nodelayer}
		\node [style=big ellipse] (0) at (-4.5, 0) {};
		\node [style=none] (1) at (-3, 0) {};
		\node [style=none] (2) at (-3, 0) {};
		\node [style=big ellipse] (3) at (-3, 0) {};
		\node [style=none] (4) at (0.25, 0) {$\vdots$};
		\node [style=vertex] (5) at (0.25, 1) {};
		\node [style=vertex] (6) at (0.25, 2) {};
		\node [style=vertex] (7) at (0.25, -1) {};
		\node [style=vertex] (8) at (0.25, -2) {};
		\node [style=vertex] (9) at (-3, -0.5) {};
		\node [style=vertex] (10) at (-3, -1) {};
		\node [style=none] (11) at (-3.75, 1.5) {\large $K_{t,t}$};

            \node[style=none] (12) at (-4.5,0.5) {};
            \node[style=none] (13) at (-4.5,-0.5) {};
            \node[style=none] (14) at (-3,0.5) {};
            \node[style=none] (15) at (-3,-0.5) {};
            \node [style=vertex] (16) at (-3, 0) {};
            \node [style=vertex] (17) at (-3, 0.5) {};
            \node [style=vertex] (18) at (-3, 1) {};
	\end{pgfonlayer}
	\begin{pgfonlayer}{edgelayer}
		\draw (9) to (6);
		\draw (9) to (5);
		\draw (9) to (7);
		\draw (9) to (8);
		\draw (10) to (6);
		\draw (10) to (5);
		\draw (10) to (7);
		\draw (10) to (8);
	\end{pgfonlayer}
        \begin{pgfonlayer}{background}
            \draw [line width=2.5pt] (12) to (14);
            \draw [line width=2.5pt] (12) to (15);
            \draw [line width=2.5pt] (13) to (14);
            \draw [line width=2.5pt] (13) to (15);
            
        \end{pgfonlayer}
\end{tikzpicture}
            \caption{Case 1}
            \label{fig:case1construction}
        \end{subfigure}
        \begin{subfigure}{0.45\textwidth}
        \centering
            \begin{tikzpicture}[scale=0.9]
	\begin{pgfonlayer}{nodelayer}
		\node [style=big circle] (0) at (-2, 0) {};
		\node [style=none] (2) at (2, 0) {$\vdots$};
		\node [style=vertex] (3) at (-1.5, -0.25) {};
		\node [style=vertex] (6) at (2, 2) {};
		\node [style=vertex] (7) at (2, 1) {};
		\node [style=vertex] (8) at (2, -2) {};
		\node [style=vertex] (9) at (2, -1) {};
		\node [style=vertex] (10) at (-1.5, 0.25) {};
            \node [style=none] (11) at (-2, 1.5) {\large $K_t$};
            \node [style=vertex] (12) at (-2.5, 0.25) {};
            \node [style=vertex] (13) at (-2.5, -0.25) {};
            \node [style=vertex] (14) at (-2, 0) {};
            \node [style=vertex] (15) at (-2, 0.75) {};
            \node [style=vertex] (16) at (-2, -0.75) {};
	\end{pgfonlayer}
	\begin{pgfonlayer}{edgelayer}
		\draw (3) to (6);
		\draw (3) to (7);
		\draw (3) to (8);
		\draw (3) to (9);
		\draw (10) to (8);
		\draw (10) to (9);
		\draw (10) to (7);
		\draw (10) to (6);
	\end{pgfonlayer}
\end{tikzpicture}
            \caption{Case 2}
            \label{fig:case2construction}
        \end{subfigure}
        \caption{Constructions of $G_n$ for $k=2$}
        \label{fig:constructions}
    \end{figure}
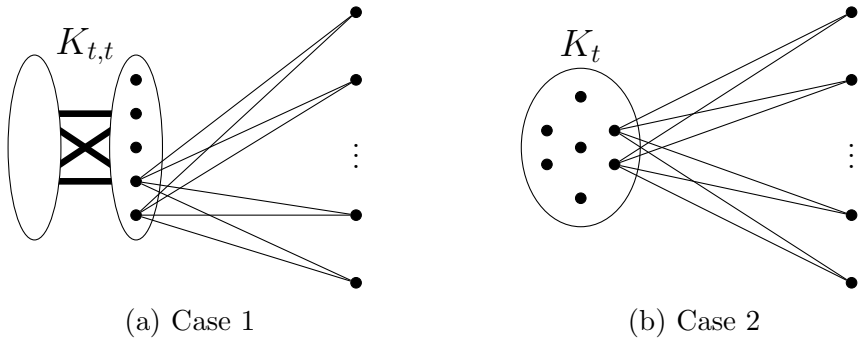

We now show that $r(G_n)\leq 2^{k}\cdot 40^k \cdot c_2f(n)$. Let $N=c_2f(n)$ and consider a red-blue colouring of the edges of the complete graph on $2^{k}\cdot 40^kN$ vertices. Start with some vertex $v_1$. Either its red or blue neighbourhood must contain at least half of the vertices of the graph. Let $R=B=\emptyset$. If the red neighbourhood of $v_1$ is larger, we add $v_1$ to $R$, otherwise we add it to $B$. We now consider only the vertices in this more popular neighbourhood of $v_1$. While $R$ and $B$ each contain less than $k$ vertices, we repeat this with vertices $v_2, v_3, \dotsc$, each time sorting $v_i$ into either $R$ or $B$ and dropping down to the larger coloured neighbourhood. We stop once one of $R$ and $B$ has size at least $k$; without loss of generality assume it is $B$. Note that this happens after at most $2k-1$ steps. We now consider the remaining vertices: if there is a vertex whose red neighbourhood is at least $\frac{1}{40}$ of the remaining graph, we add it to $R$ and drop down to its red neighbourhood. We repeat until either $|R|=k$ or we can no longer find a suitable vertex. 
    
If $|R|=k$, then there are at least $N$ vertices in $N_r(R)\cap N_b(B)$, so there is a monochromatic copy of $K_{t,t}$ which together with either $R$ or $B$ and another $n-t$ vertices from $N_r(R)\cap N_b(B)$ forms a monochromatic copy of $G_n$. Hence, we may assume that $|R|\leq k-1$ and that we end up with a set $U=N_r(R)\cap N_b(B)$ of size at least $40N$ in which every vertex has red degree strictly less than $\frac{1}{40}|U|$. Note that if there is a blue $K_{t,t}$ in $U$, then this $K_{t,t}$ together with $B$ and any other $n-t$ vertices in $U$ gives a blue $K_{t,t}$. Thus we may also assume that there is no blue $K_{t,t}$.

Let $S\subseteq U$ be an arbitrary subset of size $\frac{20}{19}t$. Let $U_0\subset U\setminus S$ be the set of vertices in $U\setminus S$ with at least $t$ blue edges to $S$.  Since there is no blue $K_{t,t,}$, the number of blue stars $K_{1,t}$ with all leaves in $S$ is at most $t\binom{|S|}{t}$; each element of $U_0$ is the centre of at least one such star, and so
$$|U_0|<t\binom{\frac{20}{19}t}{t}<2^{t/2}<N.$$
Now let $U_1=U\setminus(U_0\cup S)$.  Then $|U_1|\ge\frac{9}{10}|U|$ and every vertex of $U_1$ has at least $|S|/20$ red neighbours in $S$.  It follows that some vertex of $S$ has at least $|U_1|/20>|U|/40$ red neighbours in $U$, which contradicts our assumption on red degrees.

\medskip
    
\emph{Case 2:} $(\log n)^{k+2}n\leq f(n) \leq r(K_n)$.
Take the minimal $t$ such that $ r(K_t)> f(n)$. As before, we note that by the choice of $t$ we have $f(n)\geq r(K_{t-1})$, so $r(K_t)\leq c_1f(n)$ where $c_1$ is the constant from Lemma \ref{lemma:completegraphsbound}. Let $G_n$ be the graph constructed as follows. Take a copy of $K_t$ and an additional $n-t$ isolated vertices. Within $K_t$ choose a subset of $k$ vertices and add a complete bipartite graph between these $k$ vertices and the $n-t$ isolated vertices, see Figure \ref{fig:case2construction}. It is easy to check that $|V(G_n)|=n$, $G_n$ is $k$-connected and that $r(G_n)\geq r(K_t)> f(n)$.
    
Let $N = \ell_k c_1 f(n)$ where $\ell_k= 2\cdot  4^{k+2}\cdot k^{k+1}$. We wish to show that $r(G_n)\leq 2^{k}{\epsilon_k^{-k}}N$ where $\epsilon_k\le1/2$ is a small constant that we fix later.
    
Consider a red-blue colouring of the edges of the complete graph on $2^{k}{\epsilon_k^{-k}}N$ vertices. As before, we repeatedly pick vertices and drop down to the larger of their coloured neighbourhoods until we have $k$ vertices with edges of the same colour, say blue, to the remaining vertices. At this point, we again switch strategies and try to find vertices which have red edges to at least an $\epsilon_k$-fraction of the remaining graph. We repeat this until either we can no longer find such vertices or until we also have $k$ vertices with only red edges into the remaining set of vertices which we denote by $U$. In the latter case, we see that $|U|\geq N$, so there is a monochromatic copy of $K_t$ in $U$. This together with another $n-t$ vertices of $U$ and either the $k$ vertices that have only blue edges into $U$ or the $k$ vertices that have only red edges into $U$ gives a monochromatic copy of $G_n$. Hence, we may assume that we end up with a set $U$ of size at least $\frac{1}{\epsilon_k}N\geq 2N$ in which every vertex has red degree less than $\epsilon_k |U|$. Moreover, we can assume that $U$ does not contain a blue $K_t$.

Let $\epsilon_k$ be sufficiently small such that
\begin{equation*}
\frac{a}{\epsilon_k \log(1/\epsilon_k)}\geq 3k
\end{equation*}
where $a$ is the constant from Lemma \ref{lemma:unbalancedcolourings}. By the same lemma, we can find in $U$ a monochromatic clique of size at least 
\begin{equation*}
\frac{a}{\epsilon_k \log(1/\epsilon_k) }\log |U|\geq 3k \log(N) \geq 3k \log f(n) \geq 3k\cdot \frac{t-1}{2}\geq kt.
\end{equation*}
Call this clique $Q_1$. Since $|U\backslash V(Q_1)|\geq N$, we can repeat this and find another clique $Q_2$ of size $kt$. We greedily continue until we have $\frac{N}{kt}$ vertex-disjoint monochromatic copies of $K_{kt}$, which we label $Q_1 \dotsc, Q_{N/{kt}}$. Theses cliques cannot be blue, so they are all red. Consider the auxiliary digraph $D$ whose vertex set is $\{Q_1, \dotsc, Q_{N/kt}\}$ and whose arc set $A$ is given by the following. We have $(Q_i,Q_j)\in A$ if and only if there is some vertex $v\in V(Q_i)$ with at least $k$ red neighbours in $V(Q_j)$. Suppose there is some $Q_i$ with in-degree at least $k^kt^kn$. There are less than $k^kt^k$ subsets of $V(Q_i)$ of size $k$, so by pigeonhole principle there is some $k$-vertex subset of $Q_i$ that has at at least $n$ common red neighbours, giving us a red copy of $G_n$. Thus, we may assume that each $Q_i$ has in-degree less than $k^kt^kn$. 
Then we can use Lemma \ref{lm:digraphcolouring} to colour the $Q_i$ with at most $k^kt^kn$ colours such that there are no arcs between cliques in the same colour class.

Now some colour class has size at least $\frac{{N}/{kt}}{2k^kt^kn}\ge t$, say it contains $\{Q_1, \dotsc, Q_t\}$. Since there are no arcs between these cliques, any vertex in each of the $Q_i$ has at most $k$ red neighbours in any other $Q_j$. We now greedily embed a blue copy of $K_t$. Start with $Q_1$ and pick an arbitrary vertex $u_1$. Suppose $i\leq t$ and we have picked $\{u_1, \dotsc, u_{i-1}\}$ from cliques $\{Q_1, \dotsc, Q_{i-1}\}$ such that the vertices induce a blue $K_{t-1}$. We want to pick $v_i \in V(Q_i)$. Each of the $u_j$ has at most $k$ red neighbours in $Q_i$, so there are at least $kt-k(i-1)\geq k$ vertices in $Q_i$ with no red neighbours in $\{u_1, \dotsc, u_{i-1}\}$ and, in particular, there is some vertex $u_i$ with only blue neighbours in $\{u_1, \dotsc, u_{i-1}\}$. Reaching $i=t$, we have found a blue copy of $K_t$ and we are done.
\end{proof}

We remark here that for the range $2^{n/8}\leq f(n)\leq r(K_n)$, we can infer directly from Lemma \ref{lm:vertexdeletion} that the graph constructed in Case 2 has Ramsey number in $\Theta(f(n))$. However, as with the construction in \cite{pavez2022growthrates}, we see that this holds for a much larger range.

\section{Concluding remarks}
We have seen that for any reasonable function $f(n)$, we can find a sequence of $n$-vertex graphs $G_n$ such that $r(G_n)=\Theta(f(n))$ (where the implicit constants do not depend on $f$). It would be interesting to find out if we can, in fact, get much closer. Following Pavez-Signé et al.\ \cite{pavez2022growthrates}, let us consider the sets 
    \begin{align*}
        \mathcal{R}_n &=\{r(G)\colon |V(G)|=n\}, \text{ and} \\
        \mathcal{R}_n^c &=\{r(G)\colon |V(G)|=n \text{ and $G$ is connected}\}.
    \end{align*}

Clearly, we have $\mathcal{R}_n\subseteq [n,r(K_n)]$ and due to a result of Burr and Erd\H{o}s \cite{berdos1976trees} we know that $\mathcal{R}_n^c\subseteq [\lceil \frac{4}{3}n\rceil -1,r(K_n)]$. In order to see how close we can get to an arbitrary $f(n)$, we need to understand the `gaps' in $\mathcal{R}_n$ and $\mathcal{R}_n^c$. The following question is posed by Pavez-Signé et al.~\cite{pavez2022growthrates} and an affirmative answer would imply that it is possible to hit any growth rate {\em exactly}.

\begin{question}[\cite{pavez2022growthrates}, Question 13]
    Is there an $n_0\in \mathbb{N}$ such that for all $n\geq n_0$,
    \begin{equation*}
        \mathcal{R}_n = [n,r(K_n)] \;\; \text{and} \;\; \mathcal{R}_n^c=[\lceil \tfrac{4}{3}n\rceil -1,r(K_n)]?
    \end{equation*}
\end{question}
This is a very strong property, and the answer may well turn out to be negative.  
Pavez-Signé et al.~\cite{pavez2022growthrates} also raise the question of whether the multiplicative gaps in $\mathcal{R}_n^c$ are of size $1+o(1)$ as $n$ gets large.  It is natural to ask this even without the connectivity constraint, and we make the following conjecture.

\begin{conjecture}
\label{conj:ourconj}
    For all functions $f(n):\mathbb{N}\rightarrow \mathbb{N}$ with $n\leq f(n) \leq r(K_n)$ and for all $\epsilon>0$, there is a sequence of graphs $(G_n)_{n\in \mathbb{N}}$ satisfying $|V(G)|=n$ and 
    \begin{equation*}
        (1-\epsilon)f(n)\leq r(G_n) \leq (1+\epsilon )f(n)
    \end{equation*}
    for all sufficiently large $n$.
\end{conjecture}

One way to approach this might be through edge deletion.  Wigderson \cite{wigderson2022vtxdeletion} has conjectured that deleting an edge can change the Ramsey number by at most a constant factor.  It seems possible that, at least for dense graphs $G$, there might always some edge $e$ such that deleting $e$ changes the Ramsey number by at most a $1+o(1)$ factor (as $|G|$ gets large).

\bibliographystyle{abbrv}
\bibliography{myreferences}

\begin{thebibliography}{10}

\bibitem{Burr1973conjecture}
S.~A. Burr and P.~Erd\H{o}s.
\newblock On the magnitude of generalized {R}amsey numbers for graphs.
\newblock In {\em Infinite and finite sets ({C}olloq., {K}eszthely, 1973; dedicated to {P}. {E}rd\H{o}s on his 60th birthday), {V}ol. {I}}, volume Vol. 10 of {\em Colloq. Math. Soc. J\'{a}nos Bolyai}, pages pp 215--240. North-Holland, Amsterdam, 1975.

\bibitem{berdos1976trees}
S.~A. Burr and P.~Erd\H{o}s.
\newblock Extremal {R}amsey theory for graphs.
\newblock {\em Utilitas Math.}, 9:247--258, 1976.

\bibitem{campos2023exponential}
M.~Campos, S.~Griffiths, R.~Morris, and J.~Sahasrabudhe.
\newblock An exponential improvement for diagonal {R}amsey.
\newblock {\em arXiv preprint arXiv:2303.09521}, 2023.

\bibitem{CFS2020concentration}
D.~Conlon, J.~Fox, and B.~Sudakov.
\newblock Short proofs of some extremal results {III}.
\newblock {\em Random Structures Algorithms}, 57(4):958--982, 2020.

\bibitem{Erdos1947lowerbound}
P.~Erd\H{o}s.
\newblock Some remarks on the theory of graphs.
\newblock {\em Bull. Amer. Math. Soc.}, 53:292--294, 1947.

\bibitem{Erdos1935upperbound}
P.~Erd\H{o}s and G.~Szekeres.
\newblock A combinatorial problem in geometry.
\newblock {\em Compositio Math.}, 2:463--470, 1935.

\bibitem{Erdosszemeredi1972unbalanced}
P.~Erd\H{o}s and A.~Szemer\'{e}di.
\newblock On a {R}amsey type theorem.
\newblock {\em Period. Math. Hungar.}, 2:295--299, 1972.

\bibitem{Lee2017degeneracy}
C.~Lee.
\newblock Ramsey numbers of degenerate graphs.
\newblock {\em Ann. of Math. (2)}, 185(3):791--829, 2017.

\bibitem{pavez2022growthrates}
M.~Pavez-Sign{\'e}, S.~Piga, and N.~Sanhueza-Matamala.
\newblock Ramsey numbers with prescribed rate of growth.
\newblock {\em arXiv preprint arXiv:2209.05455}, 2022.

\bibitem{wigderson2022vtxdeletion}
Y.~Wigderson.
\newblock Ramsey numbers upon vertex deletion.
\newblock {\em arXiv preprint arXiv:2208.11181}, 2022.

\end{thebibliography}

\end{document}